\documentclass[11pt]{amsart}
\usepackage{amssymb, amsmath}
\usepackage{amsthm}
\usepackage{amscd}
\usepackage{graphicx}
\newtheorem{theorem}{Theorem}
\newtheorem{lemma}[theorem]{Lemma}

\theoremstyle{definition}

\newtheorem*{acknowledgement}{Acknowledgement}

\title{Nil happens. What about Sol?}
\author{T. T$\hat{\mathrm{a}}$m Nguy$\tilde{\hat{\mathrm{e}}}$n Phan}
\address{Department of Mathematics\\
5734 S. University Ave.\\
Chicago, IL 60637}
\email{ttamnp@math.uchicago.edu}

\DeclareMathOperator{\Nil}{Nil}

\DeclareMathOperator{\Sol}{Sol}

\def\R{\mathbb{R}}

\def\M{\overline{M}}

\input xy
\xyoption{all}
\begin{document}
\begin{abstract}
We construct complete, finite volume, $4$-dimensional manifolds with sectional curvature $-1<K<0$ with cusp cross sections compact solvmanifolds. 
\end{abstract}
\maketitle

\section{Introduction}
Let $M$ be a connected, noncompact, complete, finite volume manifold with sectional curvature $-1<K(M)<0$. By \cite{Gromovneg}, $M$ is diffeomorphic to the interior of a compact manifold $\M$ with boundary $\partial \M$. It is known that if $-1\leq K(M) \leq -a^2 <0$, then $\partial \M$ is an infra-nil manifold. For example, cross sections of cusps of finite volume, complex hyperbolic $4$-dimensional manifold are closed 3-manifolds with $\Nil$ geometry. In fact, compact flat or infra-nil manifolds can be realized as cusp cross sections of finite volume manifolds with pinched negative curvature by \cite[Corollary 6]{Ontanedahyperbolization} and \cite{BelegradekKapovich}. Fujiwara (\cite{Fujiwara}) constructed examples of $M$ for which $\partial M$ is a circle bundle over a compact hyperbolic manifold (thus is not infra-nil) for each dimension $\geq 4$.

While there are closed $3$-manifolds with $\Sol$ geometry realized as cusp cross sections of finite volume, nonpositively curved manifolds, e.g. Hilbert modular surfaces, examples of finite volume manifolds $M$ with $-1<K(M)<0$ whose cusp cross sections are diffeomorphic to a $\Sol$ manifold have not shown up (or it is unknown to the author). It might not be surprising that negatively curved cusps with $\Sol$ cross section exist because of the Fujiwara examples. We give a construction of such manifolds $M$ with $\Sol$ cusp cross sections in this paper. 

\section{The construction}
The procedure of the construction is that first we construct a negatively curved cusp with cross section a $3$-manifold with $\Sol$ geometry. Then we use Ontaneda's procedure of pinched smooth hyperbolization \cite{Ontanedahyperbolization} to glue a ``thick part" to the cusp keeping the metric smooth and negatively curved.
\begin{lemma}
Let $C$ be a compact $3$-manifold with $\Sol$ geometry. There is a complete, finite volume, negatively curved, Riemannian metric on $C\times\R$ with sectional curvature $-1<K<0$ on $(-R^2, \infty)$ for some $R^2 >0$. Moreover, this metric can be chosen so that it is a warped product metric on $C\times(0,1)$. 
\end{lemma}

\begin{proof}
The $\Sol$ metric on $C$ is 
\[ g_{\Sol} = dz^2 + e^{-2z}dx^2+ e^{2z}dy^2.\]
Consider the following metric on $C\times(0,\infty)$.
\[g = dt^2 + f(t)^2dz^2 + e^{-2t}(e^{-2z}dx^2+ e^{2z}dy^2)\]
for some function $f\colon (0,\infty) \longrightarrow \R^+$. The nonzero components (up to symmetry) of Riemann curvature tensor are
\[ R_{1212} = \dfrac{e^{-4t}(1-f(t)^2)}{f(t)^2}\;, \quad R_{1313} = e^{-2t-2z}(f(t)f'(t) -1)\;, \quad R_{1414} = -e^{-2t-2z},\]
\[ R_{2323} = e^{-2t+2z}(f(t)f'(t) -1)\;, \quad R_{2424} = -e^{-2t+2z}\;,\quad R_{3434} = -f(t)f''(t),\]
\[ R_{1431} = e^{-2t-2z}\left(1+ \dfrac{f'(t)}{f(t)}\right)\;,\quad R_{2432} = -e^{-2t+2z}\left(1+\dfrac{f'(t)}{f(t)}\right).\]
The metric $g$ has negative curvature if the following four conditions are satisfied.
\begin{itemize}
\item[a)] $f(t) >1$ for all $t$,
\item[b)] $f'(t) <0$ for all $t$, 
\item[c)] $f''(t) >0$ for all $t$, and
\item[d)] $1- f(t)f'(t) > \left( 1+ \dfrac{f'(t)}{f(t)}\right)^2$ for all $t$.
\end{itemize}
We see that $f(t) = e^{-t}$ satisfies the above four conditions for $t<0$. Also, $f(t) = 1+e^{-t}$ satisfies the above four conditions for all $t \in \R$, and the sectional curvature in this case is bounded from below. We pick the function $f(t)$ to be an interpolation between these two functions, i.e. $f(t) = e^{-t}$ for $t <<0$, and $f(t) = 1 + e^{-t}$ for $t \geq 0$, in such a way that the above four conditions are satisfied during the interpolating process. It is an exercise for the reader that such a function exists. 

We see that for $t$ negative enough the metric $g$ (with the chosen function $f$) is a warped product
\[ g = dt^2 + e^{-2t}(dz^2 + e^{-2z}dx^2+ e^{2z}dy^2).\]
It is clear that this metric is complete. Finiteness of volume follows since there is the warping factor $e^{-2t}$ decays fast enough, i.e. $\int_0^{\infty}e^{-2t}dt <\infty$.
\end{proof}

By the same procedure as in \cite[Section 11]{Ontanedahyperbolization}, e.g. apply pinched smooth hyperbolization to the suspension $\Sigma C$ of $C$, we get a finite volume, manifold with sectional curvature $-1<K(M)<0$ with two cusps, each with cross section $C$.
\begin{acknowledgement}
The author would like to thank Benson Farb for asking her the question in this paper. She should like thank the PCMI Summer School in 2012 for providing a good air-conditioned working environment during which this paper was written up. Finally she would like to thank God and the communist party of Vietnam for their great guidance.
\end{acknowledgement}
\bibliographystyle{amsplain}
\bibliography{bibliography}

\end{document}